\documentclass[a4paper,12pt]{article}

\usepackage{amsfonts,graphicx,subfig,amsmath,amssymb,enumitem,
url,authblk,amsthm,multicol,cite,float,longtable,changepage}

\DeclareMathOperator\erf{erf}

\theoremstyle{definition}

\newtheorem{theorem}{Theorem}[section]

\theoremstyle{remark}
\newtheorem{remark}{Remark}
\theoremstyle{theorem}
\newtheorem{proposition}{Proposition}[section]

\title{{\sc {\Large Simultaneous determination of two unknown thermal coefficients
through a mushy zone model with an overspecified convective boundary condition}}}

\author[1,2]{Andrea N. Ceretani\thanks{aceretani@austral.edu.ar}}
\author[1]{Domingo A. Tarzia\thanks{dtarzia@austral.edu.ar}}
\affil[1]{{\small CONICET - Depto. Matem\'atica, Facultad de Ciencias Empresariales, Universidad Austral, Paraguay 1950, S2000FZF Rosario, Argentina.}}
\affil[2]{{\small Depto. de Matem\'atica, Facultad de Ciencias Exactas, Ingenier\'ia y Agrimensura, Universidad Nacional de Rosario, Pellegrini 250, S2000BTP Rosario, Argentina.}}

\date{}

\begin{document}  
\maketitle

\begin{abstract}
The simultaneous determination of two unknown thermal coefficients for a semi-infinite material under a phase-change process with a mushy zone according to the Solomon-Wilson-Alexiades model is considered. The material is assumed to be initially liquid at its melting temperature and it is considered that the solidification process begins due to a heat flux imposed at the fixed face. The associated free boundary value problem is overspecified with a convective boundary condition with the aim of the simultaneous determination of the temperature of the solid region, one of the two free boundaries of the mushy zone and two thermal coefficients among the latent heat by unit mass, the thermal conductivity, the mass density, the specific heat and the two coefficients that characterize the mushy zone. The another free boundary of the mushy zone, the bulk temperature and the heat flux and heat transfer coefficients at the fixed face are assumed to be known. According to the choice of the unknown thermal coefficients, fifteen phase-change problems arise. The study of all of them is presented and explicit formulae for the unknowns are given, beside necessary and sufficient conditions on data in order to obtain them. Formulae for the unknown thermal coefficients, with their corresponding restrictions on data, are summarized in a table.
\end{abstract}


{\bf 2010 AMS subjet classification}: 35R35 - 35C06 - 80A22

\section{Introduction}
Heat transfer problems with a phase-change such as melting and freezing have been studied in the last century due to their wide scientific and technological applications \cite{AlSo1993,Ca1984,CaJa1959, Cr1984, Fa2005, Gu2003, Hi1987, Lu1991, Ta1986}. In particular, inverse problems related to the determination of thermal coefficients has been attracted many scientists because they are often ill-posed problems \cite{CaDu1973,BaZaFr2010,BoMh2012,DaMiUp2009,ZuLiJiGu2009,ErLeHa2013,
Go2013,HaIsLeKe2013,HaPe2013,HeNoSlWiZi2013,HuLeIv2014,InOnNaKu2007,
KaIs2012,KeIs2012,LaZh1995,LuHaJiZhWuLi2015,SaNaTa2008,PeZhLiYaGu2014,SaTa2011,
Ta1982,Ta1983,Ta1987,YaYuDe2008,YaZh2006}.

In our recent work \cite{CeTa2015-a}, we studied the determination of one unknown thermal coefficient for a semi-infinite material under a solidification process due to a heat flux imposed at the fixed boundary when the existence of a mushy zone it is considered following the Solomon-Wilson-Alexiades' model \cite{SoWiAl1982} and the associated free boundary value problem is overspecified with a convective boundary condition \cite{Ta2015}. In this paper we study the simultaneous determination of two unknown thermal coefficients for the same phase-change process when some additional information on the physical phenomena is given. 

For the convenience of the reader, we repeat the main characteristics of the phase-change process considered in \cite{CeTa2015-a,SoWiAl1982}. The material is assumed to be initially liquid at a melting temperature of $0^\circ$C and the existence of three different regions in the solidification process is considered \cite{SoWiAl1982, Ta1987}:
\begin{enumerate}
\item liquid region at temperature $T(x,t)=0$:
$D_l=\left\{(x,t)\in\mathbb{R}^2/\,x>r(t),\,t>0\right\}$,
\item solid region at temperature $T(x,t)<0$:
$D_s=\left\{(x,t)\in\mathbb{R}^2/\,0<x<s(t),\,t>0\right\}$,
\item mushy region at temperature $T(x,t)=0$:
$D_p=\left\{(x,t)\in\mathbb{R}^2/\,s(t)<x<r(t),\,t>0\right\}$,
\end{enumerate}
where $x=s(t)$ and $x=r(t)$ represent the free boundaries of the mushy zone, and $T=T(x,t)$ represents the temperature of the material. The mushy zone is considered as isothermal and the following assumptions on its structure are made:
\begin{enumerate}
\item the material contains a fixed portion of the total latent heat per unit mass (see condition (\ref{3-DC}) below),
\item its width is inversely proportional to the gradient of temperature (see condition (\ref{4-DC}) below).
\end{enumerate}
Finally, we recall that all of the thermal coefficients involved in the solidification process are assumed to be constants, where the bulk temperature $-D_\infty<0$ and the heat flux and heat transfer coefficients $q_0>0$ and $h_0>0$, respectively, at the fixed face are assumed to be known.

In this paper, we consider that we also know the evolution in time of one of the two free boundaries of the mushy zone. More precisely, we assume that the free boundary $x=s(t)$ has the form:
\begin{equation}\label{ss-DC}
s(t)=2\sigma\sqrt{t},\hspace{0.5cm}t>0\text{,}
\end{equation}
where $\sigma>0$ is a known coefficient. Thanks to this additional information on the physical phenomena, we will be able to determine simultaneously the temperature $T=T(x,t)$ of the solid region, the free boundary $x=r(t)$ and two unknown thermal coefficients among the latent heat by unit mass $l>0$, the thermal conductivity $k>0$, the mass density $\rho>0$, the specific heat $c>0$, and the two coefficients $0<\epsilon<1$ and $\gamma>0$ that characterize the mushy zone, by solving the following overspecified free boundary value problem:
\begin{align}
\label{1-DC}&\rho c T_t(x,t)-kT_{xx}(x,t)=0&0<x<s(t),\hspace{0.25cm}&t>0\\
\label{2-DC}&T(s(t),t)=0& &t>0\\
\label{3-DC}&kT_x(s(t),t)=\rho l[\epsilon\dot{s}(t)+(1-\epsilon)\dot{r}(t)]& &t>0\\
\label{4-DC}&T_x(s(t),t)(r(t)-s(t))=\gamma& &t>0\\
\label{5-DC}&r(0)=0\\
\label{6-DC}&kT_x(0,t)=\frac{q_0}{\sqrt{t}}& &t>0\\
\label{7-DC}&kT_x(0,t)=\frac{h_0}{\sqrt{t}}(T(0,t)+D_\infty)& &t>0
\end{align}

Since inverse Stefan problems are usually ill posed, it is expected that restrictions on data has to be set in order to obtain solutions to problem (\ref{1-DC})-(\ref{7-DC}). The goal of this paper is to obtain necessary and sufficient conditions on data for the fifteen phase-change problems (\ref{1-DC})-(\ref{7-DC}) that arise according to the choice of the unknown thermal coefficients, under which solutions can be obtained. Moreover, we also want to obtain those solutions explicitly.

The organization of the paper is as follows: first (Sect. \ref{sec:ExpSol-DC}) we prove a preliminary result where necessary and sufficient conditions on data for the phase-change process (\ref{1-DC})-(\ref{7-DC}) are given in order to obtain the temperature $T=T(x,t)$ and the unknown free boundary $x=r(t)$. Then (Sect. \ref{sec:ExpForm-DC}), based on this preliminary result, we present and solve the fifteen different cases for the phase-change process (\ref{1-DC})-(\ref{7-DC}) corresponding to each possible choice of the two unknown thermal coefficients among $l$, $k$, $\rho$, $c$, $\epsilon$ and $\gamma$. Under certain restrictions on data, we prove that there are twelve cases where it is possible to find a unique explicit solution and that there are infinite explicit solutions for the remainder three cases. At the end, explicit formulae for the unknown thermal coefficients, with their corresponding restrictions on data, are summarize in Table \ref{tb:1-DC}.

\section{Explicit solution to the phase-change process}\label{sec:ExpSol-DC}
The following theorem represents the base on which the work of this article will be structured.
\begin{theorem}\label{th:0-DC}
The solution to problem (\ref{1-DC})-(\ref{7-DC}) is given by:
\begin{align}
\label{sT-DC}&T(x,t)=-\frac{q_0\sqrt{\alpha t}}{k}\erf
\left(\frac{\sigma}{\sqrt{\alpha}}\right)
\left[1-
\frac{\erf\left(\frac{x}{2\sqrt{\alpha t}}\right)}
{\erf\left(\frac{\sigma}{\sqrt{\alpha}}\right)}\right]&0<x<s(t),\,&t>0\\
\label{sr-DC}&r(t)=\left[\frac{\gamma k\exp{(\sigma^2/\alpha)}}{q_0}+2\sigma
\right]\sqrt{t}& &t>0
\end{align}
if and only if the physical parameters satisfy the following two equations:
\begin{align}
\label{eq:1-DC}&\frac{q_0}{\rho l}=\left[\sigma+\frac{\gamma k (1-\epsilon)
\exp{(\sigma^2/\alpha)}}{2q_0}\right]\exp{(\sigma^2/\alpha)}\\
\label{eq:2-DC}&\erf\left(\frac{\sigma}{\sqrt{\alpha}}\right)=
\frac{kD_\infty}{q_0\sqrt{\alpha\pi}}\left(1-\frac{q_0}{h_0D_\infty}\right)
\end{align}
where the coefficient $\alpha$, defined by:
\begin{equation}
\alpha=\frac{k}{\rho c}\text{,}
\end{equation}
is the thermal diffusivity.
\end{theorem}

\begin{proof}
The free boundary value problem (\ref{1-DC})-(\ref{7-DC}) has the solution \cite{SoWiAl1982,Ta1987,TaToAppear}:
\begin{align}
\label{s1-DC}&T(x,t)=A+B\erf\left(\frac{x}{2\sqrt{\alpha t}}\right)& 0<x<s(t),\,\,&t>0\\
\label{s2-DC}&r(t)=2\mu\sqrt{\alpha t}& &t>0
\end{align}
where coefficients $A$, $B$ and $\mu$ have to be found.

By imposing conditions (\ref{2-DC})-(\ref{4-DC}), (\ref{6-DC}) and (\ref{7-DC}) on (\ref{s1-DC})-(\ref{s2-DC}) we have that:
\begin{equation}
 A=-\frac{q_0\sqrt{\alpha \pi}}{k}\erf\left(\frac{\sigma}{\sqrt{\alpha}}
 \right)\text{,}\hspace{1cm}
 B=\frac{q_0\sqrt{\alpha \pi}}{k}
 \hspace{0.5cm}\text{and} \hspace{0.5cm}
 \mu=\frac{\gamma k\exp{(\sigma^2/\alpha)}}{2q_0\sqrt{\alpha}}+
 \frac{\sigma}{\sqrt{\alpha}}\text{,}
\end{equation}
which corresponds to solution (\ref{sT-DC})-(\ref{sr-DC}),
and that the physical parameters must satisfy equations (\ref{eq:1-DC}) and (\ref{eq:2-DC}).
\end{proof}

Hence, from Theorem \ref{th:0-DC}, we have that there is an equivalence between solving the free boundary value problem (\ref{1-DC})-(\ref{7-DC}) with two unknown thermal coefficients and solving the system of equations (\ref{eq:1-DC})-(\ref{eq:2-DC}) for the same two unknown thermal coefficients as in problem (\ref{1-DC})-(\ref{7-DC}). 

\section{Explicit formulae for the unknown thermal coefficients} \label{sec:ExpForm-DC}
In this section we present and solve fifteen different cases for the phase-change process (\ref{1-DC})-(\ref{7-DC}) according to the choice of the two unknown thermal coefficients. With the aim of organizing our work, we classify each problem by making reference to the coefficients which is necessary to know in order to solve it (see Theorem \ref{th:0-DC}):

\vspace*{0.3cm}
\begin{adjustwidth}{2cm}{}
\noindent Case 1: Determination of $\epsilon$ and $\gamma$,\hspace{1cm}
Case 2: Determination of $\epsilon$ and $l$,

\vspace*{0.5cm}
\noindent Case 3: Determination of $\gamma$ and $l$,\hspace{1cm}
Case 4: Determination of $\epsilon$ and $k$,

\vspace*{0.5cm}
\noindent Case 5: Determination of $\epsilon$ and $\rho$,\hspace{1cm}
Case 6: Determination of $\epsilon$ and $c$,

\vspace*{0.5cm}
\noindent Case 7: Determination of $\gamma$ and $k$,\hspace{1cm}
Case 8: Determination $\gamma$ and $\rho$,

\vspace*{0.5cm}
\noindent Case 9: Determination of $\gamma$ and $c$,\hspace{1cm}
Case 10: Determination of $l$ and $k$,

\vspace*{0.5cm}
\noindent Case 11: Determination of $l$ and $\rho$,\hspace{1cm}
Case 12: Determination of $l$ and $c$,
\end{adjustwidth}

\begin{adjustwidth}{2cm}{}
\noindent Case 13: Determination of $k$ and $\rho$,\hspace{1cm}
Case 14: Determination of $k$ and $c$,

\vspace*{0.5cm}
\noindent and

\vspace*{0.5cm}
\noindent Case 15: Determination of $\rho$ and $c$.
\end{adjustwidth}
\vspace*{0.3cm}

Moreover, we introduce several functions and parameters which are named with an index according to the number of the case where they arise for the first time.

\begin{theorem}[Case 1: determination of $\epsilon$ and $\gamma$]\label{th:1-DC}
If we consider the phase-change process (\ref{1-DC})-(\ref{7-DC}) with unknown thermal coefficients $\epsilon$ and $\gamma$, then it has infinite  solutions given by (\ref{sT-DC})-(\ref{sr-DC}) with:
\begin{equation}\label{gamma1-DC}
\gamma=\frac{2q_0\sigma}{k(1-\epsilon)}\left(\frac{q_0}{\rho l\sigma}\exp{(-\sigma^2/\alpha)}-1\right)\exp{(-\sigma^2/\alpha)}
\end{equation}
and any $\epsilon\in(0,1)$, if and only if the remainder physical parameters satisfy condition (\ref{eq:2-DC}) and the following inequality:
\begin{equation}\label{R1-DC}
0<\frac{q_0}{\rho l\sigma}\exp{(-\sigma^2/\alpha)}-1\text{.}\tag{R1}
\end{equation}
\end{theorem}

\begin{proof}
Due to Theorem \ref{th:0-DC}, we have that the phase-change process (\ref{1-DC})-(\ref{7-DC}) has the solution given in (\ref{sT-DC})-(\ref{sr-DC}) if and only if $\epsilon$ and $\gamma$ satisfy equation (\ref{eq:1-DC}) and the remainder physical parameters satisfy condition (\ref{eq:2-DC}). Then, we have from equation (\ref{eq:1-DC}) that $\gamma$ must be given by (\ref{gamma1-DC}) for any $\epsilon\in(0,1)$. To finish the proof, only remains to observe that this coefficient $\gamma$ is positive if and only if inequality (\ref{R1-DC}) holds.
\end{proof}

\begin{theorem}[Case 2: determination of $\epsilon$ and $l$]\label{th:2-DC}
If we consider the phase-change process (\ref{1-DC})-(\ref{7-DC}) with unknown thermal coefficients $\epsilon$ and $l$, then it has infinite solutions given by (\ref{sT-DC})-(\ref{sr-DC}) with:
\begin{equation}\label{l2-DC}
l=\frac{q_0\exp{(-\sigma^2/\alpha)}}{\rho\sigma\left[1+\frac{\gamma k(1-\epsilon)}{2q_0\sigma}\exp{(\sigma^2/\alpha)}\right]}
\end{equation}
and any $\epsilon\in(0,1)$, if and only if the physical parameters $h_0$, $q_0$, $D_\infty$, $\sigma$, $\rho$, $c$ and $k$ satisfy condition (\ref{eq:2-DC}).
\end{theorem}

\begin{proof}
It is similar to the proof of Theorem \ref{th:1-DC}.
\end{proof}

\begin{theorem}[Case 3: determination of $\gamma$ and $l$]\label{th:3-DC}
If we consider the phase-change process (\ref{1-DC})-(\ref{7-DC}) with unknown thermal coefficients $\gamma$ and $l$, then it has infinite solutions given by (\ref{sT-DC})-(\ref{sr-DC}) with any $\gamma>0$ and $l$ given by (\ref{l2-DC}), if and only if the parameters $h_0$, $q_0$, $D_\infty$, $\sigma$, $\rho$, $c$ and $k$ satisfy condition (\ref{eq:2-DC}).
\end{theorem}

\begin{proof}
It is similar to the proof of Theorem \ref{th:1-DC}.
\end{proof}

\begin{remark}
Let us observe that it follows from the previous three theorems that, under certain conditions for the data of the problem, the phase-change process (\ref{1-DC})-(\ref{7-DC}) corresponding to Cases 1, 2 and 3 has an infinite number of solutions.
\end{remark}

\begin{theorem}[Case 4: determination of $\epsilon$ and $k$]\label{th:4-DC}
If we consider the phase-change process (\ref{1-DC})-(\ref{7-DC}) with unknown thermal coefficients $\epsilon$ and $k$, then it has the solution given by (\ref{sT-DC})-(\ref{sr-DC}) with:
\begin{align}
\label{epsilon4-DC}&\epsilon=1-f_4(\xi)\\
\label{k4-DC}&k=\rho c\left(\frac{\sigma}{\xi}\right)^2
\end{align}
where $\xi$ is the only positive solution to the equation:
\begin{equation}\label{eq:xi4-DC}
g_4(x)=\frac{\sigma\rho cD_\infty}{q_0\sqrt{\pi}}\left(1-\frac{q_0}{h_0D_\infty}\right)\text{,}\tag{E4}
\end{equation}
and the real functions $f_4$ and $g_4$ are defined by:
\begin{equation}\label{f4g4-DC}
f_4(x)=\frac{2q_0}{\gamma\rho c\sigma}\left(\frac{q_0}{\rho l\sigma}\exp{(-x^2)}-1\right)x^2\exp{(-x^2)}
\hspace{0.5cm}\text{and}\hspace{0.5cm}
g_4(x)=x\erf{(x)}\text{,}\hspace{0.5cm}x>0\text{,}
\end{equation}
if and only if the remainder physical parameters satisfy the next three inequalities:
\begin{align}
\label{R2-DC}&0<1-\frac{q_0}{h_0D_\infty}\tag{R2}\\
\label{R3-DC}&0<1-\frac{q_0}{\rho l \sigma}\tag{R3}\\
\label{R4-DC}&1-\frac{q_0}{h_0D_\infty}<
\frac{q_0\sqrt{\pi}}{\sigma\rho cD_\infty}
g_4\left(\sqrt{\ln{\left(\frac{q_0}{\rho l\sigma}\right)}}\right)\tag{R4}
\end{align}
and any of the following three groups of conditions:\\
Group 1: 
\begin{align}
\label{R5-DC}&f_4(\eta)>1\tag{R5}\\
\label{R6-DC}&1-\frac{q_0}{h_0D_\infty}<\frac{q_0\sqrt{\pi}}{\sigma\rho cD_\infty}g_4(\zeta_1)
\hspace{1cm}\text{or}\hspace{1cm}
1-\frac{q_0}{h_0D_\infty}>\frac{q_0\sqrt{\pi}}{\sigma\rho cD_\infty}g_4(\zeta_2)\tag{R6}
\end{align}
\hspace*{2.3cm} where $\zeta_1$ and $\zeta_2$ are the only two positive solutions to the equation:
\begin{equation}\label{eq:z1z24-DC}
f_4(x)=1\text{,}\hspace{0.5cm}
\end{equation}
\hspace*{2.3cm}and $\eta$ is the only positive solution to the equation:
\begin{equation}\label{eq:eta4-DC}
\frac{q_0}{\rho l\sigma}(1-2x^2)=(1-x^2)\exp(x^2)\text{.}
\end{equation}
Group 2:
\begin{align}
\label{R7-DC}&f_4(\eta)=1\tag{R7}\\
\label{R8-DC}&1-\frac{q_0}{h_0D_\infty}\neq\frac{q_0\sqrt{\pi}}{\sigma\rho cD_\infty}g_4(\eta)\tag{R8}
\end{align}
\hspace*{2.3cm}where $\eta$ is the only positive solution to equation (\ref{eq:eta4-DC}).\\
Group 3:
\begin{equation}\label{R9-DC}
f_4(\eta)<1\hspace{3.3cm}\tag{R9}
\end{equation}
\hspace*{2.3cm}where $\eta$ is the only positive solution to equation (\ref{eq:eta4-DC}).
\end{theorem}

\begin{proof}
As in the proof of the Theorem \ref{th:1-DC}, we have from Theorem \ref{th:0-DC} that the phase-change process (\ref{1-DC})-(\ref{7-DC}) has the solution given in (\ref{sT-DC})-(\ref{sr-DC}) if and only if $\epsilon$ and $k$ satisfy equations (\ref{eq:1-DC}) and (\ref{eq:2-DC}). Introducing the following dimensionless parameter:
\begin{equation}\label{xi-DC}
\xi=\frac{\sigma}{\sqrt{\alpha}}=\sigma\sqrt{\frac{\rho c}{k}}\text{,}
\end{equation}
we have that the solution of the system of equations (\ref{eq:1-DC})-(\ref{eq:2-DC}) is given by (\ref{epsilon4-DC})-(\ref{k4-DC}) if and only if $\xi$ is a solution to equation (\ref{eq:xi4-DC}). Then, we need to prove that the restrictions on data given in the statement are necessary and sufficient conditions for the existence of a positive solution to equation (\ref{eq:xi4-DC}) and for obtain that the coefficient $\epsilon$ given in (\ref{epsilon4-DC}) is a number between $0$ and $1$.

We first note that equation (\ref{eq:xi4-DC}) admits a positive solution if and only if inequality (\ref{R2-DC}) holds, because $g_4$ is an increasing function from $0$ to $+\infty$ in $\mathbb{R}^+$. Henceforth, we will assume that (\ref{R2-DC}) holds.

Let us now focus on the fact that $\epsilon\in(0,1)$. On one side, we have that $\epsilon$ is less than 1 if and only if (see (\ref{epsilon4-DC})): 
\begin{equation*}
0<\frac{q_0}{\rho l\sigma}\exp{(-\xi^2)}-1\text{,}
\end{equation*}
which is equivalent to inequality (\ref{R3-DC}) and:
\begin{equation}\label{R3-1-DC}
\xi<\log\left(\frac{q_0}{\rho l\sigma}\right)\text{.}
\end{equation}
Since $g_4$ is an increasing function and $\xi$ satisfies equation (\ref{eq:xi4-DC}), applying function $g_4$ side by side of inequality (\ref{R3-1-DC}) we have that it is equivalent to inequality (\ref{R4-DC}). Therefore, from now on we will assume that inequalities (\ref{R3-DC}) and (\ref{R4-DC}) also hold.

\noindent On the other side, we have that $\epsilon$ is positive if and only if (see (\ref{epsilon4-DC})):
\begin{equation}\label{ff4-DC}
f_4(\xi)<1\text{.}
\end{equation}
Let us now think of the existence of extreme values of $f_4$. We have that $f'_4(x)=0$ if and only if $x$ is a positive solution to equation (\ref{eq:eta4-DC}). In order to study equation (\ref{eq:eta4-DC}), let us introduce the real functions $u_4$ and $v_4$ defined by:
\begin{equation}\label{u4v4-DC}
u_4(x)=\frac{q_0}{\rho l\sigma}(1-2x^2)
\hspace{1cm}\text{and}\hspace{1cm}
v_4(x)=(1-x^2)\exp{(x^2)}\text{,}\hspace{0.5cm}x>0\text{.}
\end{equation}
Since $u_4$ and $v_4$ are decreasing functions such that:
\begin{align*}
&u_4(0^+)=\frac{q_0}{\rho l\sigma}>1,
\hspace{1cm}u_4(1/\sqrt{2})=0,
\hspace{1cm}u_4(+\infty)=-\infty,\\
&v_4(0^+)=1, 
\hspace{2.3cm}v_4(1)=0, 
\hspace{1.8cm}v_4(+\infty)=-\infty,
\end{align*}
it follows that equation (\ref{eq:eta4-DC}) admits only one positive solution $\eta$. Hence, $f_4$ admits only one critical point, $\eta$. Moreover, $f'_4(x)>0$ if and only if $0<x<\eta$. Then, $f_4$ has a relative maximum in $\eta$. Since $f_4(0^+)=0$ and $f_4(+\infty)=0$, it follows that $f_4$ reaches its absolute maximum in $\eta$ and that $f_4(\eta)>0$.

\noindent Now, we will study three different situations: $f(\eta)>1$, $f(\eta)=1$ and $f(\eta)<1$, which are related to conditions given in Group 1, Group 2 and Group 3, respectively.

\noindent If $f(\eta)>1$, that is, if inequality (\ref{R5-DC}) holds, we have that $\xi$ satisfies inequality (\ref{ff4-DC}) if and only if:
\begin{equation}\label{zz1zz2-DC}
\xi<\zeta_1\hspace{1cm}\text{or}\hspace{1cm}\xi>\zeta_2\text{,}
\end{equation}
where $\zeta_1$ and $\zeta_2$ are the only two positive solutions to equation (\ref{eq:z1z24-DC}). Applying the increasing function $g_4$ side by side to both inequalities and taking into account that $\xi$ satisfies equation (\ref{eq:xi4-DC}), it follows that (\ref{zz1zz2-DC}) is equivalent to (\ref{R6-DC}).

\noindent If $f(\eta)=1$, that is, if (\ref{R7-DC}) holds, we have that $\xi$ satisfies inequality (\ref{ff4-DC}) if and only if $\xi\neq\eta$. We now proceed as in the previous situation and obtain that $\xi\neq\eta$ is equivalent to (\ref{R8-DC}).

\noindent Finally, if $f(\eta)<1$, that is, if (\ref{R9-DC}) holds, we have that inequality (\ref{ff4-DC}) holds immediately.
\end{proof}

The previous Theorem \ref{th:4-DC} states necessary and sufficient conditions for the data in the overspecified free boundary value problem (\ref{1-DC})-(\ref{7-DC}) under which it is possible to find the temperature $T=T(x,t)$, the free boundary $x=r(t)$ and the two unknown thermal coefficients $\epsilon$ and $k$. Nevertheless, there are also some sufficient conditions on data, which are easier to check than the necessary and sufficient conditions given in Theorem \ref{th:3-DC}, that enable us to find the solution to problem (\ref{1-DC})-(\ref{7-DC}). Next Proposition is related to those sufficient conditions.  

\begin{proposition}[Sufficient conditions for Case 4]\label{th:4bis-DC}
Let problem (\ref{1-DC})-(\ref{7-DC}) with unknown thermal coefficients $\epsilon$ and $k$. If the remainder physical parameters satisfy inequality (\ref{R3-DC}) and the following three conditions:
\begin{align}
\label{R10-DC}&\frac{q_0\sqrt{\pi}}{\sigma\rho cD_\infty}
g_4\left(\frac{1}{\nu_4}\right)<
1-\frac{q_0}{h_0D_\infty}<
\frac{q_0\sqrt{\pi}}{\sigma\rho cD_\infty}
g_4\left(\ln{\left(\frac{q_0}{\rho l\sigma}\right)}\right)\tag{R10}\\
\label{R11-DC}&0<\frac{2q_0}{\rho\gamma c\sigma}\ln{\left(\frac{q_0}{\rho l\sigma}\right)}\left(\frac{q_0}{\rho l\sigma}-1\right)-1\tag{R11}
\end{align}
where:
\begin{equation}\label{x24-DC}
\nu_4=\frac{\rho l\sigma}{2q_0}\ln{\left(\frac{q_0}{\rho l\sigma}\right)}\left[1+\sqrt{1+\frac{2\gamma c}{l\ln\left(\frac{q_0}{\rho l\sigma}\right)}}\right]\text{,}
\end{equation}
then the solution to problem (\ref{1-DC})-(\ref{7-DC}) is given by (\ref{sT-DC})-(\ref{sr-DC}) with $\epsilon$ and $k$ given by (\ref{epsilon4-DC}) and (\ref{k4-DC}), being $\xi$ the only positive solution to the equation (\ref{eq:xi4-DC}).
\end{proposition}

\begin{proof}
Let us assume that inequalities (\ref{R3-DC}), (\ref{R10-DC}) and (\ref{R11-DC}) hold. We have seen in the proof of Theorem \ref{th:4-DC} that $\epsilon$ and $k$ must be given by (\ref{epsilon4-DC}) and (\ref{k4-DC}), and that $\xi$ must satisfy equation (\ref{eq:xi4-DC}). We have also seen that (\ref{R2-DC}) is a necessary and sufficient condition for the existence and uniqueness of the solution to equation (\ref{eq:xi4-DC}). Then, since the first inequality in (\ref{R10-DC}) implies (\ref{R2-DC}), we have that equation (\ref{eq:xi4-DC}) admits only one positive solution. Moreover, we have seen that $\epsilon$ given in (\ref{epsilon4-DC}) is less than 1 if and only if inequalities (\ref{R3-DC}) and (\ref{R4-DC}) hold. Since inequality (\ref{R10-DC}) implies inequalities (\ref{R3-DC}) and (\ref{R4-DC}), we have that $\epsilon$ given in (\ref{epsilon4-DC}) is less than 1. Finally, we have also seen that $\epsilon$ given in (\ref{epsilon4-DC}) is positive if and only if inequality (\ref{ff4-DC}) holds. The rest of the proof will be devoted to demonstrate that the first inequality in (\ref{R10-DC}) and (\ref{R11-DC}) imply inequality (\ref{ff4-DC}).

Since $\epsilon$ given in (\ref{epsilon4-DC}) is positive, we have that $\frac{q_0}{\rho l\sigma}\exp{(-\xi^2)}-1>0$, that is:
\begin{equation*}
\xi^2<\ln\left(\frac{q_0}{\rho l\sigma}\right)\text{.}
\end{equation*}
Then we have (see (\ref{f4g4-DC})):
\begin{equation}\label{ineq:4-DC}
f_4(\xi)<\frac{2q_0}{\gamma\rho c\sigma}\left(\frac{q_0}{\rho l\sigma}\exp{(-\xi^2)}-1\right)\ln\left(\frac{q_0}{\rho l\sigma}\right)\exp{(-\xi^2)}\text{.}
\end{equation}
From the above, it follows that it enough to prove that the first inequality in (\ref{R10-DC}) and (\ref{R11-DC}) imply that the right hand side of (\ref{ineq:4-DC}) is less than 1. Let $w_4$ the function defined by:
\begin{equation}\label{w4-DC}
w_4(x)=a_4x^2-b_4x-1\text{,}\hspace{0.5cm}x>0\text{,}
\end{equation}
with $a_4$ and $b_4$ given by:
\begin{equation}\label{a4b4-DC}
a_4=\frac{2q_0^2}{\rho^2\sigma^2l\gamma c}\ln\left(\frac{q_0}{\rho l\sigma}\right)>0
\hspace{1cm}\text{and}\hspace{1cm}
b_4=\frac{2q_0}{\rho\sigma\gamma c}\ln\left(\frac{q_0}{\rho l\sigma}\right)>0\text{.}
\end{equation}
We have that $\nu_4$ given in (\ref{x24-DC}) is a positive root of $w_4$. Moreover, we have that inequalities (\ref{R3-DC}) and (\ref{R11-DC}) imply $\nu_4<1$. Since the another root of $w_4$ is negative, we have that the right hand side of (\ref{ineq:4-DC}) is less than 1 if and only if $\exp{(-\xi^2)}<\nu_4$, that is, if and only if:
\begin{equation*}
\xi>\sqrt{\ln\left(\frac{1}{\nu_4}\right)}\text{.}
\end{equation*}
Only remains to observe that this last inequality is equivalent to the first inequality in (\ref{R10-DC}) because $g_4$ is an increasing function and $\xi$ satisfies equation (\ref{eq:xi4-DC}).
\end{proof}

\begin{theorem}[Case 5: determination of $\epsilon$ and $\rho$]\label{th:5-DC}
If we consider the phase-change process (\ref{1-DC})-(\ref{7-DC}) with unknown thermal coefficients $\epsilon$ and $\rho$, then it has the solution given by (\ref{sT-DC})-(\ref{sr-DC}) with:
\begin{align}
\label{epsilon5-DC}&\epsilon=1-f_5(\xi)\\
\label{rho5-DC}&\rho=\frac{k}{c}\left(\frac{\xi}{\sigma}\right)^2
\end{align}
where $\xi$ is the only positive solution to the equation:
\begin{equation}\label{eq:xi5-DC}
g_5(x)=\frac{kD_\infty}{q_0\sigma\sqrt{\pi}}\left(1-\frac{q_0}{h_0D_\infty}\right)\tag{E5}
\end{equation}
and the real functions $f_5$ and $g_5$ are defined by:
\begin{equation}\label{f5g5-DC}
f_5(x)=\frac{2q_0\sigma}{\gamma k}\left(\frac{q_0 c\sigma}{lk}\,\,\frac{\exp{(-x^2)}}{x^2}-1\right)\exp{(-x^2)}
\hspace{1cm}\text{and}\hspace{1cm}
g_5(x)=\frac{\erf{(x)}}{x}\text{,}\hspace{0.5cm}x>0\text{,}
\end{equation}
if and only if the remainder physical parameters satisfy the following two conditions:
\begin{equation}\label{R12-DC}
\frac{q_0\sigma\sqrt{\pi}}{kD_\infty}g_5(\zeta_1)<
1-\frac{q_0}{h_0D_\infty}<
\min\left\{\frac{2q_0\sigma}{kD_\infty},\frac{q_0\sigma\sqrt{\pi}}{kD_\infty}g_5(\zeta_2)\right\}\text{,}\tag{R12}
\end{equation}
where $\zeta_1$ and $\zeta_2$ are, respectively, the only positive solutions to equations:
\begin{align}
\label{eq:zeta15-DC}&\frac{q_0\sigma c}{lk}\exp{(-x^2)}=x^2\\
\label{eq:zeta25-DC}&\frac{q_0\sigma c}{lk}\exp{(-x^2)}=\left[\frac{\gamma k}{2q_0\sigma}\exp{(x^2)}+1\right]x^2\text{.}
\end{align}
\end{theorem}

\begin{proof}
As in the previous proofs, we have from Theorem \ref{th:0-DC} that the phase-change process (\ref{1-DC})-(\ref{7-DC}) has the solution given in (\ref{sT-DC})-(\ref{sr-DC}) if and only if $\epsilon$ and $\rho$ are given by (\ref{epsilon5-DC}) and (\ref{rho5-DC}), where the dimensionless parameter $\xi$ (see (\ref{xi-DC})) is a solution to equation (\ref{eq:xi5-DC}). Then, we need to prove that the two restrictions given by (\ref{R12-DC}) are necessary and sufficient conditions for the existence of a positive solution to equation (\ref{eq:xi5-DC}) and for obtain that the coefficient $\epsilon$ given in (\ref{epsilon5-DC}) is a number between $0$ and $1$.

Since $g_5$ is a decreasing function from $\frac{2}{\sqrt{\pi}}$ to $0$ in $\mathbb{R}^+$, we have that equation (\ref{eq:xi5-DC}) admits positive solutions if and only if:
\begin{equation}\label{R12-1-DC}
0<1-\frac{q_0}{h_0D_\infty}<\frac{2q_0\sigma}{kD_\infty}\text{.}
\end{equation}

Let us assume for a moment that (\ref{R12-1-DC}) holds and focus on the fact that $\epsilon\in(0,1)$. As we did in the proof of the Theorem \ref{th:4-DC}, let us first think about the existence of extreme values of $f_5$. We have that $f'_5(x)=0$ if and only if:
\begin{equation}\label{ff5-DC}
u_5(x)=x^2\text{,}
\end{equation}
where $u_5$ is the real function defined by:
\begin{equation}\label{u5-DC}
u_5(x)=\frac{q_0\sigma c}{lk}\left(1+\frac{1}{x^2}\right)\exp{(-x^2)}\text{,}\hspace{0.5cm}x>0\text{.}
\end{equation}
Since $u_5$ is an increasing function from $0$ to $+\infty$ in $\mathbb{R^+}$, it follows that equation (\ref{ff5-DC}) admits only one positive solution $\eta$. Then $f_5$ has only one critical point, $\eta$. Moreover, $f'_5(x)>0$ if and only if $x>\eta$. Then, $f_5$ has a relative minimum in $\eta$. Since $f_5(0^+)=+\infty$ and $f_5(+\infty)=0$, it follows that $f_5$ reaches its absolute minimum in $\eta$ and that $f_5(\eta)<0$. Then we have that the coefficient  $\epsilon$ given in (\ref{epsilon5-DC}) is a number between 0 and 1 if and only if $\zeta_1<\xi<\zeta_2$, where $\zeta_1$ and $\zeta_2$ are the only two positive numbers such that $f_5(\zeta_1)=0$ and $f_5(\zeta_2)=1$, that is, the only positive solutions to equations (\ref{eq:zeta15-DC}) and (\ref{eq:zeta25-DC}), respectively. Since $g_4$ is a decreasing function and $\xi$ satisfies equation (\ref{eq:xi5-DC}), we have that $\zeta_1<\xi<\zeta_2$ is equivalent to:
\begin{equation}\label{R12-2-DC}
\frac{q_0\sigma\sqrt{\pi}}{kD_\infty}g_5(\zeta_1)<1-\frac{q_0}{h_0D_\infty}<\frac{q_0\sigma\sqrt{\pi}}{kD_\infty}g_5(\zeta_1)\text{.}
\end{equation}
Only remains to observe that inequalities (\ref{R12-1-DC}) and (\ref{R12-2-DC}) are equivalent to the inequalities given by (\ref{R12-DC}).
\end{proof}

\begin{theorem}[Case 6: determination of $\epsilon$ and $c$]\label{th:6-DC}
If we consider the phase-change process (\ref{1-DC})-(\ref{7-DC}) with unknown thermal coefficients $\epsilon$ and $c$, then it has the solution given by (\ref{sT-DC})-(\ref{sr-DC}) with:
\begin{align}
\label{epsilon6-DC}&\epsilon=1-f_6(\xi)\\
\label{c6-DC}&c=\frac{k}{\rho}\left(\frac{\xi}{\sigma}\right)^2
\end{align}
where $\xi$ is the only positive solution to the equation (\ref{eq:xi5-DC}),
and $f_6$ is the real function defined by:
\begin{equation}\label{f6-DC}
f_6(x)=\frac{2q_0\sigma}{\gamma k}\left(\frac{q_0}{\rho l\sigma}\exp{(-x^2)}\right)\exp{(-x^2)}\text{,}\hspace{0.5cm}x>0\text{,}
\end{equation}
if and only if the remainder physical parameters satisfy the following condition:
\begin{equation}\label{R13-DC}
\frac{q_0\sigma\sqrt{\pi}}{kD_\infty}f_5\left(\sqrt{\ln{\left(\frac{q_0}{\rho l\sigma}\right)}}\right)<
1-\frac{q_0}{h_0D_\infty}\tag{R13}
\end{equation}
and any of the following two groups of conditions:\\
Group 1:
\begin{align}
\label{R14-DC}&\frac{q_0}{\rho l\sigma}\geq\frac{\gamma k}{2q_0\sigma}+1\tag{R14}\\
\label{R15-DC}&1-\frac{q_0}{h_0D_\infty}<
\min\left\{\frac{2q_0\sigma}{kD_\infty},\frac{q_0\sigma\sqrt{\pi}}{kD_\infty}f_5\left(\sqrt{\ln{\left(\frac{1}{\nu_6}\right)}}\right)
\right\}\tag{R15}
\end{align}
\hspace*{3.6cm} where:
\begin{equation}\label{x26-DC}
\nu_6=\frac{\rho l\sigma}{2q_0}\left[1+\sqrt{1+\frac{2\gamma k}{\sigma^2\rho l}}\right]\text{.}
\end{equation}
Group 2:
\begin{align}
\label{R16-DC}&1<\frac{q_0}{\rho l\sigma}<\frac{\gamma k}{2q_0\sigma}+1\tag{R16}\\
\label{R17-DC}&1-\frac{q_0}{h_0D_\infty}<
\frac{2q_0\sigma}{kD_\infty}\tag{R17}\text{.}
\end{align}
\end{theorem}

\begin{proof}
It is similar to the proof of Theorem \ref{th:4-DC}.
\end{proof}

\begin{theorem}[Case 7: determination of $\gamma$ and $k$]\label{th:7-DC}
If we consider the phase-change process (\ref{1-DC})-(\ref{7-DC}) with unknown thermal coefficients $\gamma$ and $k$, then it has the solution given by (\ref{sT-DC})-(\ref{sr-DC}) with:
\begin{equation}\label{gamma7-DC}
\gamma=\frac{2q_0}{\sigma\rho c(1-\epsilon)}\left(\frac{q_0}{\rho l\sigma}\exp{(-\xi^2)}-1\right)\xi^2\exp{(-\xi^2)}
\end{equation}
and $k$ given by (\ref{k4-DC}),
where $\xi$ is the only positive solution to the equation (\ref{eq:xi4-DC}),
if and only if the physical parameters $h_0$, $q_0$, $D_\infty$, $\sigma$, $l$, $\rho$ and $c$ satisfy conditions (\ref{R2-DC}), (\ref{R3-DC}) and (\ref{R4-DC}).
\end{theorem}

\begin{proof}
Once again, we have from the Theorem \ref{th:0-DC} that $\gamma$ and $k$ must be given by (\ref{gamma7-DC}) and (\ref{k4-DC}), where the dimensionless parameter $\xi$ (see (\ref{xi-DC})) is a solution to equation (\ref{eq:xi4-DC}). As we saw in the proof of the Theorem \ref{th:4-DC}, the equation (\ref{eq:xi4-DC}) admits positive solutions if and only if inequality (\ref{R2-DC}) holds. 

To complete the proof only remains to observe that the coefficient $\gamma$ given in (\ref{gamma7-DC}) is positive if and only if $0<\frac{q_0}{\rho l\sigma}\exp{(-\xi^2)}-1$ which, as we also saw in the proof of Theorem \ref{th:4-DC}, is equivalent to inequalities (\ref{R3-DC}) and (\ref{R4-DC}).
\end{proof}

\begin{theorem}[Case 8: determination of $\gamma$ and $\rho$]\label{th:8-DC}
If we consider the phase-change process (\ref{1-DC})-(\ref{7-DC}) with unknown thermal coefficients $\gamma$ and $\rho$, then it has the solution given by (\ref{sT-DC})-(\ref{sr-DC}) with:
\begin{equation}\label{gamma8-DC}
\gamma=\frac{2q_0\sigma}{k(1-\epsilon)}\left(\frac{q_0c\sigma}{lk}\,\,\frac{\exp{(-\xi^2)}}{\xi}-1\right)\exp{(-\xi^2)}
\end{equation}
and $\rho$ given by (\ref{rho5-DC}),
where $\xi$ is the only positive solution to the equation (\ref{eq:xi5-DC}),
if and only if the physical parameters $h_0$, $q_0$, $D_\infty$, $\sigma$, $k$ and $c$ satisfy conditions (\ref{R2-DC}), (\ref{R17-DC}) and:
\begin{equation}\label{R18-DC}
g_5(\eta)<\frac{kD_\infty}{q_0\sigma\sqrt{\pi}}\left(1-\frac{q_0}{h_0D_\infty}\right)\tag{R18}
\end{equation}
where $g_5$ is defined in (\ref{f5g5-DC})
and $\eta$ is the only positive solution to the equation:
\begin{equation}\label{eq:eta8-DC}
\frac{q_0c\sigma}{lk}\frac{exp{(-x^2)}}{x}=1\text{.}
\end{equation}
\end{theorem}

\begin{proof}
From the Theorem \ref{th:0-DC}, we have that $\gamma$ and $\rho$ must be given by (\ref{gamma8-DC}) and (\ref{rho5-DC}), where the dimensionless parameter $\xi$ (see (\ref{xi-DC})) is a solution to equation (\ref{eq:xi5-DC}). As we saw in the proof of the Theorem \ref{th:5-DC}, the equation (\ref{eq:xi5-DC}) admits a positive solution if and only if inequality (\ref{R12-1-DC}) holds, that is, if and only if inequalities (\ref{R2-DC}) and (\ref{R17-DC}) hold.

We also have that the coefficient $\gamma$ given in (\ref{gamma8-DC}) is positive if and only if:
\begin{equation}\label{uu8-DC}
u_8(\xi)>0\text{,}
\end{equation}
where $u_8$ is the real function defined by:
\begin{equation}\label{u8-DC}
u_8(x)=\frac{q_0 c\sigma}{lk}\,\,\frac{\exp{(-x^2)}}{x}-1\text{,}\hspace{0.5cm}x>0\text{.}
\end{equation}
Since $u_8$ is a decreasing function from $+\infty$ to $-1$ in $\mathbb{R}^+$, we have that equation (\ref{eq:eta8-DC}) has only one positive solution $\eta$. Therefore, inequality (\ref{uu8-DC}) holds if and only if $\xi<\eta$. Only remains to observe that $\xi<\eta$ is equivalent to condition (\ref{R18-DC}) because $g_5$ is a decreasing function and $\xi$ satisfies equation (\ref{eq:xi5-DC}).
\end{proof}

\begin{theorem}[Case 9: determination of $\gamma$ and $c$]\label{th:9-DC}
If we consider the phase-change process (\ref{1-DC})-(\ref{7-DC}) with unknown thermal coefficients $\gamma$ and $c$, then it has the solution given by (\ref{sT-DC})-(\ref{sr-DC}) with:
\begin{equation}\label{gamma9-DC}
\gamma=\frac{2q_0\sigma}{k(1-\epsilon)}\left(\frac{q_0}{\rho l\sigma}\exp{(-\xi^2)}-1\right)\exp{(-\xi^2)}
\end{equation}
and $c$ given by (\ref{c6-DC}),
where $\xi$ is the only positive solution to the equation (\ref{eq:xi5-DC}),
if and only if the physical parameters $h_0$, $q_0$, $D_\infty$, $\sigma$, $l$, $k$ and $\rho$ satisfy conditions (\ref{R3-DC}), (\ref{R13-DC}) and (\ref{R17-DC}).
\end{theorem}

\begin{proof}
It is similar to the proof of Theorem \ref{th:8-DC}.
\end{proof}

\begin{theorem}[Case 10: determination of $l$ and $k$]\label{th:10-DC}
If we consider the phase-change process (\ref{1-DC})-(\ref{7-DC}) with unknown thermal coefficients $l$ and $k$, then it has the solution given by (\ref{sT-DC})-(\ref{sr-DC}) with:
\begin{equation}\label{l10-DC}
l=\frac{q_0}{\rho\sigma}\left[\frac{1}{1+\frac{\rho c\sigma\gamma(1-\epsilon)}{2q_0}}\,\,\frac{\exp{(\xi^2)}}{\xi^2}\right]\exp{(-\xi^2)}
\end{equation}
and $k$ given by (\ref{k4-DC}),
where $\xi$ is the only positive solution to the equation (\ref{eq:xi4-DC}),
if and only if the physical parameters $h_0$, $q_0$ and $D_\infty$ satisfy condition (\ref{R2-DC}).
\end{theorem}

\begin{proof}
It is similar to the proof of Theorem \ref{th:7-DC}.
\end{proof}

\begin{theorem}[Case 11: determination of $l$ and $\rho$]\label{th:11-DC}
If we consider the phase-change process (\ref{1-DC})-(\ref{7-DC}) with unknown thermal coefficients $l$ and $\rho$, then it has the solution given by (\ref{sT-DC})-(\ref{sr-DC}) with:
\begin{equation}\label{l11-DC}
l=\frac{q_0c\sigma}{k}\left[\frac{1}{1+\frac{\gamma k(1-\epsilon)}{2q_0\sigma}\exp{(\xi^2)}}\right]\frac{\exp{(-\xi^2)}}{\xi}
\end{equation}
and $\rho$ given by (\ref{rho5-DC}),
where $\xi$ is the only positive solution to the equation (\ref{eq:xi5-DC}),
if and only if the physical parameters $h_0$, $q_0$, $D_\infty$ and $k$ satisfy conditions (\ref{R2-DC}) and (\ref{R17-DC}).
\end{theorem}

\begin{proof}
It is similar to the proof of Theorem \ref{th:8-DC}.
\end{proof}

\begin{theorem}[Case 12: determination of $l$ and $c$]\label{th:12-DC}
If we consider the phase-change process (\ref{1-DC})-(\ref{7-DC}) with unknown thermal coefficients $l$ and $c$, then it has the solution given by (\ref{sT-DC})-(\ref{sr-DC}) with:
\begin{equation}\label{l12-DC}
l=\frac{q_0}{\rho\sigma}\left[\frac{1}{1+\frac{\gamma k(1-\epsilon)}{2q_0\sigma}\exp{(\xi^2)}}\right]\exp{(-\xi^2)}
\end{equation}
and $c$ given by (\ref{c6-DC}),
where $\xi$ is the only positive solution to the equation (\ref{eq:xi5-DC}),
if and only if the physical parameters $h_0$, $q_0$, $D_\infty$ and $k$ satisfy conditions (\ref{R2-DC}) and (\ref{R17-DC}).
\end{theorem}

\begin{proof}
It is similar to the proof of Theorem \ref{th:8-DC}.
\end{proof}

\begin{theorem}[Case 13: determination of $k$ and $\rho$]\label{th:13-DC}
If we consider the phase-change process (\ref{1-DC})-(\ref{7-DC}) with unknown thermal coefficients $k$ and $\rho$, then it has the solution given by (\ref{sT-DC})-(\ref{sr-DC}) with:
\begin{align}
\label{k13-DC}&k=\frac{q_0\sigma\sqrt{\pi}}{D_\infty\left(1-\frac{q_0}{h_0D_\infty}\right)}g_5(\xi)\\
\label{rho13-DC}&\rho=\frac{q_0\sqrt{\pi}}{c\sigma D_\infty}g_4(\xi)
\end{align}
where real functions $g_4$ and $g_5$ are defined in (\ref{f4g4-DC}) and (\ref{f5g5-DC}), $\xi$ is the only positive solution to the equation:
\begin{equation}\label{eq:xi13-DC}
g_{13}(x)=\frac{a_{13}}{c_{13}}h_{13}(x)\text{,}\tag{E13}
\end{equation}
and the real functions $g_{13}$ and $h_{13}$ are defined by:
\begin{equation}\label{g13h13-DC}
g_{13}(x)=\frac{\exp{(-x^2)}}{\erf{(x)}}
\hspace{1cm}\text{and}\hspace{1cm}
h_{13}(x)=x+b_{13}\exp{(x^2)}\erf{(x)}\text{,}\hspace{0.5cm}x>0\text{,}
\end{equation}
with:
\begin{equation}\label{a8b8c8-DC}
a_{13}=\frac{2cD_\infty}{l\sqrt{\pi}}\left(1-\frac{q_0}{h_0D_\infty}\right)^2\text{,}\hspace{1cm}
b_{13}=\frac{\gamma\sqrt{\pi}(1-\epsilon)}{2D_\infty\left(1-\frac{q_0}{h_0D_\infty}\right)}\text{,}
c_{13}=2\left(1-\frac{q_0}{h_0D_\infty}\right)\text{,}
\end{equation}
if and only if the physical parameters $h_0$, $q_0$ and $D_\infty$ satisfy condition (\ref{R2-DC}).
\end{theorem}

\begin{proof}
We have from Theorem \ref{th:0-DC} that the phase-change process (\ref{1-DC})-(\ref{7-DC}) has the solution given in (\ref{sT-DC})-(\ref{sr-DC}) if and only if $k$ and $\rho$ satisfy equations (\ref{eq:1-DC}) and (\ref{eq:2-DC}). Since inequality (\ref{R2-DC}) is a necessary condition for the existence of a solution to equation (\ref{eq:2-DC}), now and on we will assume that inequality (\ref{R2-DC}) holds. 

We have that the solution to the system of equations (\ref{eq:1-DC})-(\ref{eq:2-DC}) is given by (\ref{k13-DC}) and (\ref{rho13-DC}), where the dimensionless parameter $\xi$ is a solution to equation (\ref{eq:xi13-DC}). Only remains to observe that equation (\ref{eq:xi13-DC}) admits a positive solution since $g_{13}$ is a decreasing function from $+\infty$ to $0$ in $\mathbb{R}^+$ and $h_{13}$ is an increasing function from $0$ to $+\infty$ in $\mathbb{R}^+$.
\end{proof}

\begin{theorem}[Case 14: determination of $k$ and $c$]\label{th:14-DC}
If we consider the phase-change process (\ref{1-DC})-(\ref{7-DC}) with unknown thermal coefficients $k$ and $c$, then it has the solution given by (\ref{sT-DC})-(\ref{sr-DC}) with $k$ given by (\ref{k13-DC}) and:
\begin{equation}\label{c14-DC}
c=\frac{q_0\sqrt{\pi}}{\sigma\rho D_\infty\left(1-\frac{q_0}{h_0D_\infty}\right)}g_4(\xi)\text{,}
\end{equation}
where the real function $g_4$ is defined in (\ref{f4g4-DC}), $\xi$ is the only positive solution to the equation:
\begin{equation}\label{eq:xi14-DC}
a_{14}g_{14}(x)=h_{14}(x)\tag{E14}
\end{equation}
and the real functions $g_{14}$ and $h_{14}$ are defined by:
\begin{equation}\label{g14h14-DC}
g_{14}(x)=\left(\frac{q_0}{\rho l\sigma}\exp{(-x^2)}-1\right)x
\hspace{1cm}\text{and}\hspace{1cm}
h_{14}(x)=\erf{(x)}\exp{(x^2)}\text{,}\hspace{0.5cm}x>0\text{,}
\end{equation}
with:
\begin{equation}\label{a14-DC}
a_{14}=\frac{2D_\infty}{\gamma\sqrt{\pi}(1-\epsilon)}\left(1-\frac{q_0}{h_0D_\infty}\right)\text{,}
\end{equation}
if and only if the remainder physical parameters satisfy conditions (\ref{R2-DC}), (\ref{R3-DC}) and:
\begin{equation}\label{R19-DC}
g_{14}(\eta)>h_{14}(\eta)\tag{R19}\text{,}
\end{equation}
where $\eta$ is the only positive solution to the equation:
\begin{equation}\label{eq:eta14-DC}
\frac{q_0}{\rho l\sigma}(1-2x^2)=\exp{(x^2)}\text{.}
\end{equation}
\end{theorem}

\begin{proof}
It is similar to the proof of Theorem \ref{th:13-DC}.
\end{proof}

\begin{theorem}[Case 15: determination of $\rho$ and $c$]\label{th:15-DC}
If we consider the phase-change process (\ref{1-DC})-(\ref{7-DC}) with unknown thermal coefficients $\rho$ and $c$, then it has the solution given by (\ref{sT-DC})-(\ref{sr-DC}) with:
\begin{align}
\label{rho15-DC}&\rho=\frac{q_0}{l\sigma}\frac{\exp{(-\xi^2)}}{1+\frac{\gamma k(1-\epsilon)}{2q_0\sigma}\exp{(\xi^2)}}\\
\label{c15-DC}&c=\frac{kl}{\sigma q_0}\left[1+\frac{\gamma k(1-\epsilon)}{2q_0\sigma}\exp{(\xi^2)}\right]\xi^2\exp{(\xi^2)}
\end{align}
where $\xi$ is the only positive solution to the equation (\ref{eq:xi5-DC}),
if and only if the physical parameters $h_0$, $q_0$, $D_\infty$, $\sigma$ and $k$ satisfy conditions (\ref{R2-DC}) and (\ref{R17-DC}).
\end{theorem}

\begin{proof}
It is similar to the proof of Theorem \ref{th:13-DC}.
\end{proof}

Next table summarizes the results of this section corresponding to the fifteen cases for the simultaneous determination of two unknown thermal coefficients for problem (\ref{1-DC})-(\ref{7-DC}).

\begin{center}
\begin{longtable}{ clcc }
\hline
Case& \hspace*{1.5cm} Thermal coefficients & Equation for & Restrictions on data\\
& & $\xi=\frac{\sigma}{\sqrt{\alpha}}=\sigma\sqrt{\frac{\rho c}{k}}$ & \\
\endfirsthead
\hline
Case& \hspace*{1.5cm} Thermal coefficients & Equation for & Restrictions on data\\
& & $\xi=\frac{\sigma}{\sqrt{\alpha}}=\sigma\sqrt{\frac{\rho c}{k}}$ & \\
\hline
\\
\endhead
\hline
\\
1&
$0<\epsilon<1$&
 --- &
(\ref{eq:2-DC}), (\ref{R1-DC})\\
 &
$\gamma=\frac{2q_0\sigma}{k(1-\epsilon)}\left(\frac{q_0}{\rho l\sigma}\exp{(-\sigma^2/\alpha)}-1\right)\exp{(-\sigma^2/\alpha)}$&
 & \\
 \\
2&
$0<\epsilon<1$&
 --- &
(\ref{eq:2-DC})\\
 &
$l=\frac{q_0\exp{(-\sigma^2/\alpha)}}{\rho\sigma\left[1+\frac{\gamma k(1-\epsilon)}{2q_0\sigma}\exp{(\sigma^2/\alpha)}\right]}$&
 &\\
 \\
3&
$\gamma>0$&
 --- &
(\ref{eq:2-DC})\\
 &
$l=\frac{q_0\exp{(-\sigma^2/\alpha)}}{\rho\sigma\left[1+\frac{\gamma k(1-\epsilon)}{2q_0\sigma}\exp{(\sigma^2/\alpha)}\right]}$&
 &\\
 \\
 4&
 $\epsilon=1-f_4(\xi)$, with $f_4$ given in (\ref{f4g4-DC})&
 (\ref{eq:xi4-DC})&
 (\ref{R2-DC}), (\ref{R3-DC}), (\ref{R4-DC}), (\ref{R5-DC}), (\ref{R6-DC})\\
 &
 $k=\rho c\left(\frac{\sigma}{\xi}\right)^2$&
 &
or\\
 &
 &
 &
(\ref{R2-DC}), (\ref{R3-DC}), (\ref{R4-DC}), (\ref{R7-DC}), (\ref{R8-DC})\\  
 & &
 & or\\
 & & &(\ref{R2-DC}), (\ref{R3-DC}), (\ref{R4-DC}), (\ref{R9-DC})   \\
 \\
5&
$\epsilon=1-f_5(\xi)$, with $f_5$ given in (\ref{f5g5-DC})&
(\ref{eq:xi5-DC})&
(\ref{R12-DC})\\
 &
$\rho=\frac{k}{c}\left(\frac{\xi}{\sigma}\right)^2$&
 & \\
 \\
 6&
$\epsilon=1-f_6(\xi)$, with $f_6$ given in (\ref{f6-DC})&
(\ref{eq:xi5-DC})&
(\ref{R13-DC}), (\ref{R14-DC}), (\ref{R15-DC})\\ 
 &
 $c=\frac{k}{\rho}\left(\frac{\xi}{\sigma}\right)^2$&
 &
 or\\
 && &(\ref{R13-DC}), (\ref{R16-DC}), (\ref{R17-DC})\\
 \\
7&
$\gamma=\frac{2q_0}{\sigma\rho c(1-\epsilon)}\left(\frac{q_0}{\rho l\sigma}\exp{(-\xi^2)}-1\right)\xi^2\exp{(-\xi^2)}$&
(\ref{eq:xi4-DC})&
(\ref{R2-DC}), (\ref{R3-DC}), (\ref{R4-DC})\\
 &
$k=\rho c\left(\frac{\sigma}{\xi}\right)^2$&
 &
 \\
 \\
8&
$\gamma=\frac{2q_0\sigma}{k(1-\epsilon)}\left(\frac{q_0c\sigma}{lk}\,\,\frac{\exp{(-\xi^2)}}{\xi}-1\right)\exp{(-\xi^2)}$&
(\ref{eq:xi5-DC})&
(\ref{R2-DC}), (\ref{R17-DC}), (\ref{R18-DC})\\
 &
$\rho=\frac{k}{c}\left(\frac{\xi}{\sigma}\right)^2$&
 & \\
 \\
9&
$\gamma=\frac{2q_0\sigma}{k(1-\epsilon)}\left(\frac{q_0}{\rho l\sigma}\exp{(-\xi^2)}-1\right)\exp{(-\xi^2)}$&
(\ref{eq:xi5-DC})&
(\ref{R3-DC}), (\ref{R13-DC}), (\ref{R17-DC})\\
 &
$c=\frac{k}{\rho}\left(\frac{\xi}{\sigma}\right)^2$&
 & \\
 \\
10&
$l=\frac{q_0}{\rho\sigma}\left[\frac{1}{1+\frac{\rho c\sigma\gamma(1-\epsilon)}{2q_0}}\,\,\frac{\exp{(\xi^2)}}{\xi^2}\right]\exp{(-\xi^2)}$&
(\ref{eq:xi4-DC})&
(\ref{R2-DC})\\
 &
$k=\rho c\left(\frac{\sigma}{\xi}\right)^2$&
 & \\
 \\
11&
$l=\frac{q_0c\sigma}{k}\left[\frac{1}{1+\frac{\gamma k(1-\epsilon)}{2q_0\sigma}\exp{(\xi^2)}}\right]\frac{\exp{(-\xi^2)}}{\xi}$&
(\ref{eq:xi5-DC})&
(\ref{R2-DC}), (\ref{R17-DC})\\
 &
$\rho=\frac{k}{c}\left(\frac{\xi}{\sigma}\right)^2$&
 & \\
 \\
12&
$l=\frac{q_0}{\rho\sigma}\left[\frac{1}{1+\frac{\gamma k(1-\epsilon)}{2q_0\sigma}\exp{(\xi^2)}}\right]\exp{(-\xi^2)}$&
(\ref{eq:xi5-DC})&
(\ref{R2-DC}), (\ref{R17-DC})\\
 &
$c=\frac{k}{\rho}\left(\frac{\xi}{\sigma}\right)^2$&
 & \\
 \\
13&
$k=\frac{q_0\sigma\sqrt{\pi}}{D_\infty\left(1-\frac{q_0}{h_0D_\infty}\right)}g_5(\xi)$, with $g_5$ given in (\ref{f5g5-DC})&
(\ref{eq:xi13-DC})&
(\ref{R2-DC})\\
 &
$\rho=\frac{q_0\sqrt{\pi}}{c\sigma D_\infty}g_4(\xi)$, with $g_4$ given in (\ref{f4g4-DC})&
 & \\
 \\
 \\
14&
$k=\frac{q_0\sigma\sqrt{\pi}}{D_\infty\left(1-\frac{q_0}{h_0D_\infty}\right)}g_5(\xi)$, with $g_5$ given in (\ref{f5g5-DC})&
(\ref{eq:xi14-DC})&
(\ref{R2-DC}), (\ref{R3-DC}), (\ref{R19-DC})\\
 &
$c=\frac{q_0\sqrt{\pi}}{\sigma\rho D_\infty\left(1-\frac{q_0}{h_0D_\infty}\right)}g_4(\xi)$, with $g_4$ given in (\ref{f4g4-DC})&
 &\\
 \\
 \\
 \\
 \\
15&
$\rho=\frac{q_0}{l\sigma}\frac{\exp{(-\xi^2)}}{1+\frac{\gamma k(1-\epsilon)}{2q_0\sigma}\exp{(\xi^2)}}$&
(\ref{eq:xi5-DC})&
(\ref{R2-DC}), (\ref{R17-DC})\\
 &
$c=\frac{kl}{\sigma q_0}\left[1+\frac{\gamma k(1-\epsilon)}{2q_0\sigma}\exp{(\xi^2)}\right]\xi^2\exp{(\xi^2)}$&
 &\\
 \\
\hline
\caption{{\sc Formulae for problem (\ref{1-DC})-(\ref{7-DC})}. Explicit formulae for the two unknown thermal coefficients among $l$, $\gamma$, $\epsilon$, $k$, $\rho$ or $c$, with the corresponding equation for the dimensionless parameter $\xi$ (see (\ref{xi-DC})) and restrictions on data.}
\label{tb:1-DC}
\end{longtable}
\end{center}

\section{Conclusions}
In this paper, we consider a semi-infinite material under a solidification process caused by an initial heat flux boundary condition, when the thermophysical parameters involved in the phase-change process are assumed to be constant and the Solomon-Wilson-Alexiades' mushy zone model it is considered. We suppose that evolution in time of one of the two free boundaries of the mushy zone, the bulk temperature and the heat flux and heat transfer coefficients at the fixed face are known. We overspecify the associated free boundary value problem with the aim of the simultaneous determination of the temperature of the material, the unknown free boundary and two unknown thermal coefficients among the latent heat by unit mass, the thermal conductivity, the mass density, the specific heat and the two coefficients that characterize the mushy zone. We first prove a preliminary result where necessary and sufficient conditions on data for the phase-change process are given in order to obtain the temperature and the unknown free boundary. Then, based on this preliminary result, we present and solve the fifteen different cases for the phase-change process corresponding to each possible choice of the two unknown thermal coefficients. We prove that, under certain restrictions on data, there are twelve cases where it is possible to find a unique explicit solution and that there are infinite explicit solutions for the remainder three cases. For each case, we give formulae for the temperature, the unknown free boundary and the two unknown thermal coefficients, beside the necessary and sufficient conditions on data in order to obtain them. We summarize explicit formulae for the two unknown thermal coefficients in Table \ref{tb:1-DC}.

\section*{Competing interests}
The authors declare that there is no conflict of interests regarding the publication of this paper.

\section*{Acknowledgements}
This paper has been partially sponsored by the Project PIP No. 0534 from CONICET-UA (Rosario, Argentina) and AFOSR-SOARD Grant FA 9550-14-1-0122.
\bibliographystyle{plain}
\bibliography{References}

\appendix
\section*{Appendix: Summary of restrictions on data, definition of functions and equations for $\xi$.}
\subsection*{List of restrictions on data}
\begin{align}
&0<\frac{q_0}{\rho l\sigma}\exp{(-\sigma^2/\alpha)}-1\tag{R1}\hspace{5.5cm}\\
&0<1-\frac{q_0}{h_0D_\infty}\tag{R2}\hspace{5.5cm}\\
&0<\frac{q_0}{\rho l \sigma}-1\tag{R3}\hspace{5.5cm}\\
&1-\frac{q_0}{h_0D_\infty}<
\frac{q_0\sqrt{\pi}}{\sigma\rho cD_\infty}
g_4\left(\sqrt{\ln{\left(\frac{q_0}{\rho l\sigma}\right)}}\right)\tag{R4}\hspace{5.5cm}\\
&f_4(\eta)>1\tag{R5}\hspace{5.5cm}
\end{align}
\hspace*{3.5cm} where $\eta$ is the only positive solution to the equation:
\begin{equation*}
\frac{q_0}{\rho l\sigma}(1-2x^2)=(1-x^2)\exp(x^2)\text{,}\hspace{0.5cm}x>0\hspace{2.4cm}
\end{equation*}
\begin{align}
&1-\frac{q_0}{h_0D_\infty}<\frac{q_0\sqrt{\pi}}{\sigma\rho cD_\infty}g_4(\zeta_1)
\hspace{1cm}\text{or}\hspace{1cm}
1-\frac{q_0}{h_0D_\infty}>\frac{q_0\sqrt{\pi}}{\sigma\rho cD_\infty}g_4(\zeta_2)\tag{R6}
\end{align}
\hspace*{3.5cm} where $\zeta_1$ and $\zeta_2$ are the only two positive solutions to the equation:\\
\begin{equation*}
f_4(x)=1\text{,}\hspace{0.5cm}x>0\hspace{6.5cm}
\end{equation*}
\begin{align*}
&f_4(\eta)=1\tag{R7}\hspace{10.3cm}
\end{align*}
\hspace*{3.5cm} where $\eta$ is the only positive solution to the equation (\ref{eq:eta4-DC}).
\begin{align}
&1-\frac{q_0}{h_0D_\infty}\neq\frac{q_0\sqrt{\pi}}{\sigma\rho cD_\infty}g_4(\eta)\tag{R8}\hspace{7.1cm}
\end{align}
\hspace{3.5cm}where $\eta$ is the only positive solution to the equation (\ref{eq:eta4-DC}).
\begin{align}
&f_4(\eta)<1\hspace{3.3cm}\tag{R9}\hspace{7.1cm}
\end{align}
\hspace{3.5cm}where $\eta$ is the only positive solution to the equation (\ref{eq:eta4-DC}).
\begin{align}
&\frac{q_0\sqrt{\pi}}{\sigma\rho cD_\infty}
g_4\left(\sqrt{\ln\left(\frac{1}{\nu_4}\right)}\right)<
1-\frac{q_0}{h_0D_\infty}<
\frac{q_0\sqrt{\pi}}{\sigma\rho cD_\infty}
g_4\left(\sqrt{\ln{\left(\frac{q_0}{\rho l\sigma}\right)}}\right)\tag{R10}
\end{align}
\hspace*{3.5cm} where:
\begin{equation*}
\nu_4=\frac{\rho l\sigma}{2q_0}\ln{\left(\frac{q_0}{\rho l\sigma}\right)}\left[1+\sqrt{1+\frac{2\gamma c}{l\ln\left(\frac{q_0}{\rho l\sigma}\right)}}\right]
\end{equation*}
\begin{align}
&0<\frac{2q_0}{\rho\gamma c\sigma}\ln{\left(\frac{q_0}{\rho l\sigma}\right)}\left(\frac{q_0}{\rho l\sigma}-1\right)-1\tag{R11}\hspace{5.9cm}\\
&\frac{q_0\sigma\sqrt{\pi}}{kD_\infty}g_5(\zeta_1)<
1-\frac{q_0}{h_0D_\infty}<
\min\left\{\frac{2q_0\sigma}{kD_\infty},\frac{q_0\sigma\sqrt{\pi}}{kD_\infty}g_5(\zeta_2)\right\}\tag{R12}\hspace{2.1cm}
\end{align}
\hspace*{3.5cm} where $\zeta_1$ and $\zeta_2$ are, respectively, the unique only solutions to equations:
\begin{align*}
&\frac{q_0\sigma c}{lk}\exp{(-x^2)}=x^2\text{,}\hspace{0.5cm}x>0\hspace{1cm}\\
&\frac{q_0\sigma c}{lk}\exp{(-x^2)}=\left[\frac{\gamma k}{2q_0\sigma}\exp{(x^2)}+1\right]x^2\text{,}\hspace{0.5cm}x>0\hspace{1cm}
\end{align*}
\begin{align}
&\frac{q_0\sigma\sqrt{\pi}}{kD_\infty}g_5\left(\sqrt{\ln{\left(\frac{q_0}{\rho l\sigma}\right)}}\right)<
1-\frac{q_0}{h_0D_\infty}\tag{R13}\hspace{3cm}\\
&\frac{q_0}{\rho l\sigma}\geq\frac{\gamma k}{2q_0\sigma}+1\tag{R14}\hspace{3cm}\\
&1-\frac{q_0}{h_0D_\infty}<
\min\left\{\frac{2q_0\sigma}{kD_\infty},\frac{q_0\sigma\sqrt{\pi}}{kD_\infty}g_5\left(\sqrt{\ln{\left(\frac{1}{\nu_6}\right)}}\right)
\right\}\tag{R15}\hspace{3cm}
\end{align}
\hspace*{3.5cm} where:
\begin{equation*}
\nu_6=\frac{\rho l\sigma}{2q_0}\left[1+\sqrt{1+\frac{2\gamma k}{\sigma^2\rho l}}\right]
\end{equation*}
\begin{align}
&1<\frac{q_0}{\rho l\sigma}<\frac{\gamma k}{2q_0\sigma}+1\tag{R16}\hspace{7cm}\\
&1-\frac{q_0}{h_0D_\infty}<
\frac{2q_0\sigma}{kD_\infty}\tag{R17}\hspace{7cm}\\
&g_5(\eta)<\frac{kD_\infty}{q_0\sigma\sqrt{\pi}}\left(1-\frac{q_0}{h_0D_\infty}\right)\tag{R18}\hspace{7cm}
\end{align}
\hspace*{3.5cm} where $\eta$ is the only positive solution to the equation:
\begin{equation*}
\frac{q_0c\sigma}{lk}\frac{exp{(-x^2)}}{x}=1\text{,}\hspace{0.5cm}x>0\hspace{4.7cm}
\end{equation*}
\begin{align}
&g_{14}(\eta)>h_{14}(\eta)\tag{R19}\hspace{9.5cm}
\end{align}
\hspace*{3.5cm} where $\eta$ is the only positive solution to the equation:
\begin{equation*}
\frac{q_0}{\rho l\sigma}(1-2x^2)=\exp{(x^2)}\text{,}\hspace{0.5cm}x>0\hspace{4cm}
\end{equation*}

\subsection*{List of definitions of functions ($x>0$)}
\begin{align*}
&f_4(x)=\frac{2q_0}{\gamma\rho c\sigma}\left(\frac{q_0}{\rho l\sigma}\exp{(-x^2)}-1\right)x^2\exp{(-x^2)}\text{,}
\hspace{0.8cm}
g_4(x)=x\erf{(x)}\text{,}\\
&f_5(x)=\frac{2q_0\sigma}{\gamma k}\left(\frac{q_0 c\sigma}{lk}\,\,\frac{\exp{(-x^2)}}{x^2}-1\right)\exp{(-x^2)}\text{,}
\hspace{1cm}
g_5(x)=\frac{\erf{(x)}}{x}\text{,}\\
&f_6(x)=\frac{2q_0\sigma}{\gamma k}\left(\frac{q_0}{\rho l\sigma}\exp{(-x^2)}\right)\exp{(-x^2)}\text{,}\\
&g_{13}(x)=\frac{\exp{(-x^2)}}{\erf{(x)}}\text{,}
\hspace{5.1cm}
h_{13}(x)=x+b_{13}\exp{(x^2)}\erf{(x)}\text{,}
\end{align*}
\hspace*{0.1cm} with:
\begin{equation*}
a_{13}=\frac{2cD_\infty}{l\sqrt{\pi}}\left(1-\frac{q_0}{h_0D_\infty}\right)^2\text{,}\hspace{1cm}
b_{13}=\frac{\gamma\sqrt{\pi}(1-\epsilon)}{2D_\infty\left(1-\frac{q_0}{h_0D_\infty}\right)}\text{,}
c_{13}=2\left(1-\frac{q_0}{h_0D_\infty}\right)\text{,}
\end{equation*}
\begin{align*}
&g_{14}(x)=\left(\frac{q_0}{\rho l\sigma}\exp{(-x^2)}-1\right)x\text{,}
\hspace{3.1cm}
h_{14}(x)=\erf{(x)}\exp{(x^2)}\text{,}\hspace{1.7cm}
\end{align*}
\hspace*{0.1cm} with:
\begin{equation*}
a_{14}=\frac{2D_\infty}{\gamma\sqrt{\pi}(1-\epsilon)}\left(1-\frac{q_0}{h_0D_\infty}\right)\text{.}
\end{equation*}

\subsection*{List of equations for $\xi$}
\begin{align}
&g_4(x)=\frac{\sigma\rho cD_\infty}{q_0\sqrt{\pi}}\left(1-\frac{q_0}{h_0D_\infty}\right)\tag{E4}\\
&g_5(x)=\frac{kD_\infty}{q_0\sigma\sqrt{\pi}}\left(1-\frac{q_0}{h_0D_\infty}\right)\tag{E5}\\
&g_{13}(x)=\frac{a_{13}}{c_{13}}h_{13}(x)\tag{E13}\\
&a_{14}g_{14}(x)=h_{14}(x)\tag{E14}
\end{align}

\end{document}